\documentclass{amsart} 
\usepackage{graphicx}
\usepackage{newlattice}

\DeclareMathOperator{\New}{Prot}
\DeclareMathOperator{\Prot}{Prot}
\DeclareMathOperator{\Btw}{Mid[S]}

\newcommand{\fp}{\F p}
\newcommand{\fq}{\F q}

\newcommand{\lcorner}{\textup{lc}} 
\newcommand{\rcorner}{\textup{rc}}

\newcommand{\oa}{\ol \bga}

\theoremstyle{plain}
\newtheorem{theorem}{Theorem}

\newtheorem{lemma}[theorem]{Lemma}
\newtheorem{corollary}[theorem]{Corollary}

\theoremstyle{definition}

\newtheorem{problem}{Problem}

\begin{document}
\title[Congruences of fork extensions]{Congruences of fork extensions of\\ slim, planar, semimodular lattices}  
\author{G. Gr\"{a}tzer} 
\address{Department of Mathematics\\
  University of Manitoba\\
  Winnipeg, MB R3T 2N2\\
  Canada}
\email[G. Gr\"atzer]{gratzer@me.com}
\urladdr[G. Gr\"atzer]{http://server.maths.umanitoba.ca/homepages/gratzer/}

\date{Feb. 17, 2014}
\subjclass[2010]{Primary: 06C10. Secondary: 06B10.}
\keywords{semimodular lattice, fork extension, congruence, join-irreducible congruence, prime interval, planar, slim.}

\begin{abstract}
For a slim, planar, semimodular lattice $L$ and covering square~$S$, 
G.~Cz\'edli and E.\,T.~Schmidt introduced the fork extension, $L[S]$,
which is also a slim, planar, semimodular lattice.
We investigate when a congruence of $L$ extends to $L[S]$. 

We introduce a join-irreducible congruence 
$\boldsymbol{\gamma}(S)$ of $L[S]$.
We determine when it is new, 
in the sense that it is not generated 
by a join-irreducible congruence of $L$. 
When it is new, we describe the congruence $\boldsymbol{\gamma}(S)$ 
in great detail.
The main result follows: 
\emph{In the order of join-irreducible congruences 
of a slim, planar, semimodular lattice $L$,
the congruence $\boldsymbol{\gamma}(S)$ has \emph{at most two covers.}}
\end{abstract}

\maketitle

\section{Introduction}\label{S:Introduction}
    
Let $L$ be a planar semimodular lattice.
As in G. Gr\"atzer and E.~Knapp~\cite{GKn07},  
we call $L$ \emph{slim} if $L$ contains
no $\SM 3$ sublattice.
Note an alternative definition of slimness in G. Cz\'edli and E.\,T. Schmidt~\cite{CS12}; it implies planarity.

Let $L$ be a slim, planar, semimodular lattice, an \emph{SPS lattice}.
As in G.~Cz\'edli and E.\,T.~Schmidt~\cite{CS12a}, 
\emph{inserting a fork} into $L$ 
at a covering square~$S = \set{o, a_l, a_r, t}$ 
($a_l$~to the left of $a_r$), 
firstly, replaces~$S$ by a~copy of $\SfS 7$, 
adding the elements $b_l, b_r, m$ as in Figure~\ref{F:s7}. 

Secondly, if there is a chain 
$u\prec v\prec w$ such that~the element $v$ has just been inserted 
and $T = \set{x=u \mm z, z, u, w=z \jj u}$
is a covering square in the lattice~$L$ 
(and so $u \prec v \prec w$ is not on the boundary of $L$) 
but $x \prec z$ at the present stage of the construction,
then we insert a new element~$y$ 
such that $x \prec y \prec z$ and $y \prec v$.
Let $L[S]$ denote the lattice we obtain at the termination of the process.
As~observed in G.~Cz\'edli and E.\,T.~Schmidt~\cite{CS12a},
$L[S]$ is an SPS lattice.
See Figure~\ref{F:forkexample} for an illustration.

In this paper, we start the investigation of the congruences of $L[S]$
as they relate to the congruences of $L$. 

Let $L$ and $K$ be lattices and let $K$ be an extension of~$L$.
Let $\bga$ be a congruence of~$L$ and let $\bgb$ be a congruence of~$K$.
We call $\bgb$ an \emph{extension} of $\bga$
if $\bgb$ restricted to~$L$, in formula, $\bgb\restr L$, equals $\bga$.
Let $\oa = \consub{K}{\bga}$ 
denote the smallest congruence of~$K$ containing $\bga$; it is called the congruence of $L[S]$ \emph{generated by $\bga$}. 
In general, $\oa$ is not an extension of $\bga$. 
If it is, then we call $\oa$ the \emph{minimal extension} 
of $\bga$ to~$K$.  

If $\bga$ is a join-irreducible congruence of $L$
(that is, $\bga= \con{\fp}$ for a prime interval $\fp$ of $L$),
then $\oa$ is a join-irreducible congruence of $L[S]$ (namely, $\ol\bga= \consub{L[S]}{\fp}$).

Every prime interval of $L[S]$ is perspective to a prime interval of $L$, 
so the only candidate for a new join-irreducible congruence in $L[S]$ is the congruence $\bgg(S) = \consub{L[S]}{m,t}$.

\begin{figure}[t]
\centerline{\includegraphics{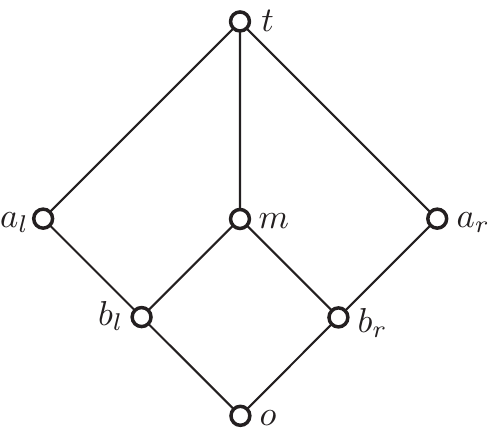}}
\caption{The lattice $\SfS 7$}\label{F:s7}
\end{figure}

\begin{figure}[t]
\centerline{\includegraphics{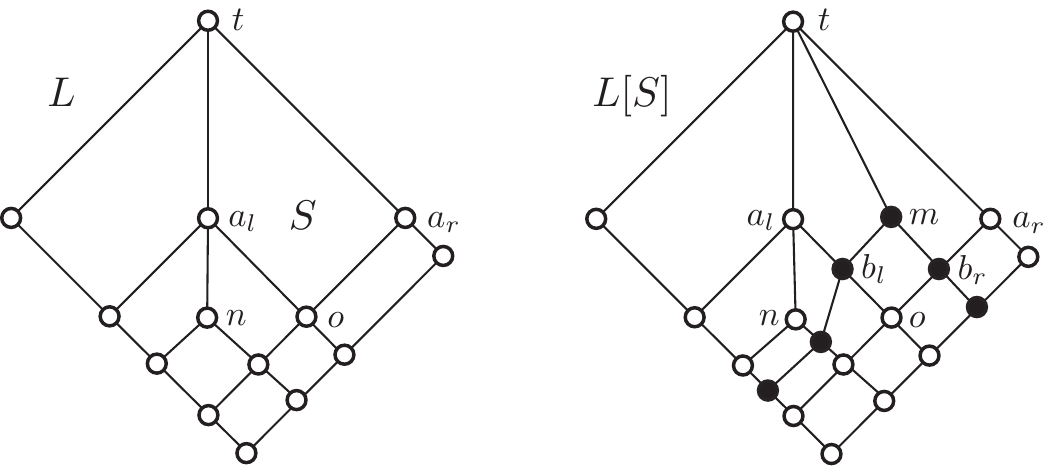}}
\caption{Inserting a fork at $S$}\label{F:forkexample}
\end{figure}

We examine the extendibility of congruences to a fork extension.

\begin{theorem}\label{T:extension}
Let $L$ be a slim, semimodular, planar lattice 
and $\bga$ be a congruence of~$L$. 
Let $S = \set{o, a_l, a_r, t}$ be a covering square of $L$.
\begin{enumeratei}
\item If $\bga \restr S = \one_S$, 
then $\bga$ extends to $L[S]$.
\item If $\bga \restr S = \zero_S$, 
then $\bga$ extends to $L[S]$.
\item If $\bga \restr S$ is not trivial, 
then $\bga$ may or may not extend to $L[S]$.
\end{enumeratei}
\end{theorem}

Let $L$ and $S$ be as in Theorem~\ref{T:extension}. We call the covering square of $S =\set{o, a_l, a_r, t}$  a \emph{tight square}, 
if $t$ covers exactly two elements, namely, $a_l$ and $a_r$, in $L$;
otherwise, $S$ is a \emph{wide square}.

\begin{theorem}\label{T:wide}
Let $L$ be an SPS lattice. 
If $S$ is a wide square, 
then $\bgg(S) = \consub{L[S]}{m,t}$ is generated by a~congruence of $L$.
\end{theorem}

\begin{theorem}\label{T:tight}
Let $L$ be an SPS lattice. 
Let $S =\set{o, a_l, a_r, t}$  be a \emph{tight square}. 
Then $L[S]$ has exactly one join-irreducible congruence, 
namely $\bgg(S) = \consub{L[S]}{m,t}$,
that is \emph{not generated} by a~congruence of $L$.
\end{theorem}

We now state the most important property of $\bgg(S)$:

\begin{theorem}\label{T:uppercover}
Let $S$ be a tight square in an SPS lattice $L$. 
Then the congruence $\bgg(S)$ of $L[S]$ is covered 
by one or two congruences in the order of join-irreducible
congruences of $L[S]$, namely,
by $\consub{L[S]}{a_l, t}$ and $\consub{L[S]}{a_r, t}$.
\end{theorem}

In Section \ref{S:Congruences}, we recall some concepts and results
on congruences of finite lattices we need in this paper. 
Some basic facts about SPS lattices are stated in Section~\ref{S:SPS}.
We introduce the notation 
for the fork construction in Section~\ref{S:forks}.
In Section~\ref{S:Extending}, 
we analyze which congruences of $L$ extend to $L[S]$.
In~ Section~\ref{S:wide}, we verify Theorem~\ref{T:wide}.
In describing $\bgg(S)$, \emph{protrusions} on a fork cause the problems;
these are introduced in Section~\ref{S:Protrusion}.

In Section \ref{S:noprotrusions}, we describe $\bgg(S)$ provided that
$S$ has no protrusions.
We describe $\bgg(S)$ on a part of $L[S]$ in Section~\ref{S:gammaK}.
Utilizing these result, in Section~\ref{S:gammaL}
we verify Theorem~\ref{T:tight}.
Finally, in Section~\ref{S:uppercovers}, 
applying the description of $\bgg(S)$ developed 
in Section~\ref{S:Protrusion}, we prove Theorem~\ref{T:uppercover}.

In Section~\ref{S:comments}, we state some open problems.

We will use the notations and concepts of lattice theory, 
as in \cite{LTF}.

We will use the traditional approach to planarity: 
a planar lattice is a lattice with a planar diagram, unspecified.
A number of recent papers (especially by G.~Cz\'edli) 
use a more rigorous approach. 
For an overview of this new approach, 
see G. Cz\'edli and G. Gr\"atzer~\cite{CGa}.
For example, Lemma~\ref{L:cproj} uses the concept of ``adjacency'';
this seems to be diagram dependent, but in the context it is not.

\section{Congruences of lattices}\label{S:Congruences}

As~illustrated in Figure~\ref{F:cong2}, 
we say that $[a,b]$ is \emph{up congruence-perspective} 
to $[c,d]$ and write $[a,b] \cperspup [c,d]$\label{GoN:cperspup} 
if $a \leq c$ and $d = b \jj c$; similarly,
$[a,b]$ is \emph{down congruence-perspective} to $[c,d]$ 
and write $[a,b] \cperspdn [c,d]$ if $d \leq b$ and $c = a \mm d$.
If $[a,b] \cperspup [c,d]$ \emph{or} $[a,b] \cperspdn [c,d]$, 
then $[a,b]$ is \emph{congruence-perspective} to $[c,d]$ 
and we write $[a,b] \cpersp [c,d]$.

\begin{figure}[tbh]
  \centerline{\includegraphics{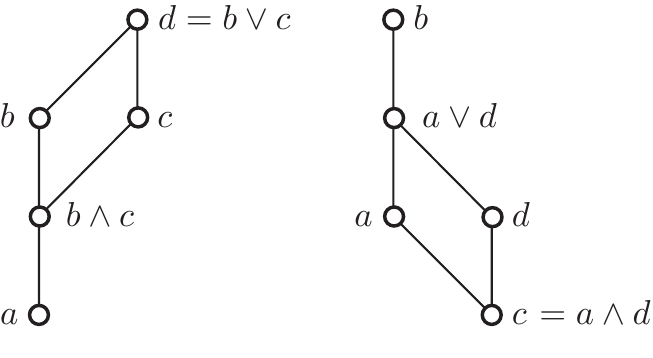}}
  \caption{$[a,b] \cperspup [c, d]$ and $[a,b] \cperspdn [c, d]$}\label{F:cong2}
\end{figure}

If for some natural number~$n$ 
and intervals $[e_i, f_i]$, for $0 \leq i \leq n$,
\[
   [a,b] = [e_0,f_0] \cpersp [e_1,f_1] \cpersp 
      \cdots \cpersp [e_n,f_n] = [c,d],
\]
then we call $[a,b]$ \emph{congruence-projective}
to $[c,d]$, and we write $[a, b] \cproj [c, d]$. 

We now state a classic result in J.~Jakubik~\cite{jJ55} in a special case, 
see also, \cite[Lemma 238]{LTF}.

\begin{lemma}\label{L:cproj}
Let L be a finite lattice, $a \leq b$ in $L$, 
and let $\fq$ be a prime interval.
Then $\fq$ is collapsed by $\con{a, b}$
if{f} $[a, b] \cproj \fq$.
In fact, there is a prime interval $\fp$ in $[a,b]$ 
such that $\fp \cproj \fq$.
\end{lemma}

The following technical lemma, 
see G. Gr\"atzer \cite{gG13f}, 
plays a crucial role in the computations in this paper.

\begin{lemma}\label{L:technical}
Let $L$ be a finite lattice. 
Let $\bgd$ be an equivalence relation on $L$
with intervals as equivalence classes.
Then $\bgd$ is a congruence relation if{}f the following condition 
and its dual hold:
\begin{equation}\label{E:cover}
\text{For $a \prec b$,  $a \prec c$, and $b \neq c$ in $L$,
if $a \equiv c\,(\tup{mod}\, \bgd)$,
then $c \equiv b \jj c\,(\tup{mod}\,\bgd)$.}\tag{C${}_{\jj}$}
\end{equation}
\end{lemma}

We denote by (C${}_{\mm}$) the dual of (C${}_{\jj}$).

\section{SPS lattices}\label{S:SPS}
The following statements can be found in the literature 
(see G. Gr\"atzer and E.~Knapp~\cite{GKn07}--\cite{GKn10}, 
G.~Cz\'edli and E.\,T. Schmidt~\cite{CS12}--\cite{CS12a}); 
for a survey of this field see G. Cz\'edli and G. Gr\"atzer \cite{CGa}
and G. Gr\"atzer \cite{gG14a}.

\begin{lemma}\label{L:known}
Let $L$ be an SPS lattice. 
\begin{enumeratei}
\item An element of $L$ has at most two covers.

\item Let $a \in L$. 
Let $a$ cover the three elements $x_1$, $x_2$, and $x_3$.
Then the set $\set{x_1,x_2,x_3}$ generates an $\SfS 7$ sublattice.

\item If the elements $x_1$, $x_2$, and $x_3$ are adjacent, 
then the $\SfS 7$ sublattice of \tup{(ii)} is a cover-preserving sublattice.
\end{enumeratei}
\end{lemma}

Finally, we state the \emph{Structure Theorem for SPS Lattices} of
G. Cz\'edli and E.\,T. Schmidt~\cite{CS12}:

\begin{theorem}\label{T:Structure Theorem}
Let $L$ be an SPS lattice. 
There exists a planar distributive lattice~$D$
such that $L$ can be obtained from $D$ by a series 
of fork insertions.
\end{theorem}

Of course, a planar distributive lattice is 
just a sublattice of a direct product of~two chains.

For a planar lattice $L$, define 
a \emph{left corner} (resp., \emph{right corner}) 
as a doubly-irreducible element in $L - \set{0,1}$ 
on the left (resp., right) boundary of~$L$. 
G. Gr\"atzer and E. Knapp~\cite{GKn09} 
define a \emph{rectangular lattice} $L$ 
as a planar semimodular lattice 
which has exactly one left corner, $\lcorner(L)$, 
and exactly one right corner, $\rcorner(L)$, 
and they are complementary---that is, 
$\lcorner(L) \jj \rcorner(L) = 1$ 
and $\lcorner(L) \mm \rcorner(L) = 0$. 

Let us call a rectangular lattice $L$ a \emph{patch lattice} 
if $\lcorner(A)$ and $\rcorner(A)$ are dual atoms. 
The lattice $\SfS 7$, see Figure~\ref{F:s7}, is an example of a slim patch lattice.

\begin{corollary}\label{C:Structure Theorem}
Let $L$ be a patch lattice.
Then $L$ can be obtained from the four-element Boolean lattice, $\SC 2^2$, by a series of fork insertions.
\end{corollary}

\section{The fork construction}\label{S:forks}

Let $L$ be an SPS lattice. 
Let $S =\set{o, a_l, a_r, t}$ be a covering square of $L$,
let $a_l$ be to the left of $a_r$
We need some notation for the $L[S]$ construction, 
see Figure~\ref{F:forkdetails}. 
 
Two prime intervals of $L$ are \emph{consecutive}
if they are opposite sides of a covering square. 
As in G.~Cz\'edli and E.\,T.\ Schmidt~\cite{CS12},
a maximal sequence of  
consecutive prime intervals form a \emph{trajectory}. 
So a trajectory is an equivalence class of the transitive reflexive closure of the ``consecutive'' relation. 

Consider the trajectory containing (determined by) a prime interval $\fp$; the part of the trajectory to the left of $\fp$ (including $\fp$)
is called the \emph{left wing} of $\fp$; 
see G.~Cz\'edli and G. Gr\"atzer \cite{CG12}. 
The left wing of $[o,a_l]$ is also called the \emph{left wing} of
the covering square $S =\set{o, a_l, a_r, t}$.
We define the \emph{right wing} symmetrically. Note that trajectories start and end in prime intervals on the boundary. So a left wing starts with
a prime interval on the boundary. 

\begin{figure}[h]
\centerline{\includegraphics{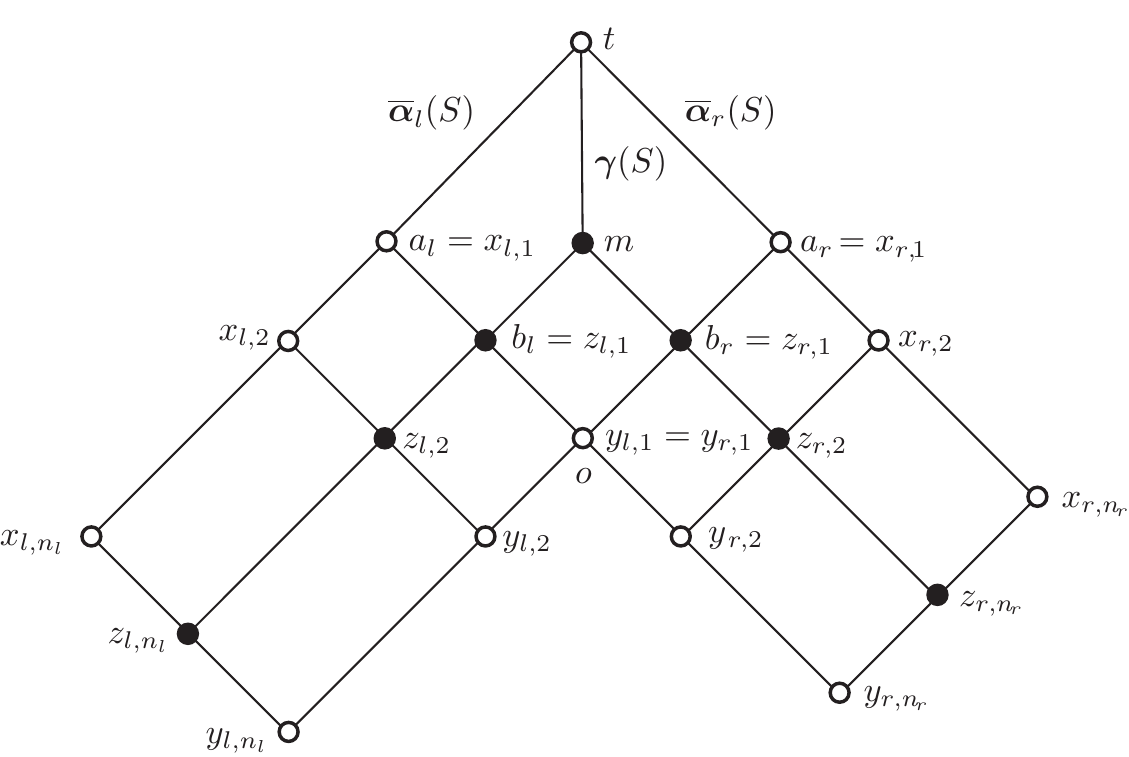}}
\caption{Notation for the fork construction}\label{F:forkdetails}
\end{figure}

Let $a_l = x_{l,1}$, $o = y_{l,1}$, and let
\begin{equation}\label{E:leftwing}
   [y_{l,1}, x_{l,1}], [y_{l,2}, x_{l,2}], \dots, [y_{l,n_l}, x_{l,n_l}]
\end{equation} 
be the left wing of $S$. Note that $[y_{l,n_l}, x_{l,n_l}]$ is on the left boundary and these prime intervals in Figure~\ref{E:leftwing}
form a sublattice of the form $\SC 2 \times \SC {n_l}$.

We add the elements $z_{l,1}$, $z_{l,2}$, \dots, $z_{l,n_l}$ so that 
$y_{l,i} \prec z_{l,i} \prec x_{l,i}$, for $i = 1, \dots, n_l$ and 
$z_{l,n_l} \prec \dots \prec z_{l,2} \prec z_{l,1}$. 
So now we have a sublattice of the form $\SC 3 \times \SC n$.
We proceed symmetrically on the right with $n_r$ pairs.

Finally, we add the element $m$,
so that the set $\set{o, b_l, b_r, a_l, a_r, m, t}$ 
forms an $\SfS 7$ sublattice. Note that $m = b_l \jj b_r$. 

It is easy to compute that we obtain an extension $L[S]$ of $L$.
Let
\[
   F[S] = \set{m,z_{l,1} \succ \dots \succ z_{l,n_l}, 
          z_{r,1} \succ \dots \succ z_{r,m_r}}
\]
be the set of new elements;
they are black filled in Figure~\ref{F:forkdetails}.  

\begin{lemma}\label{L:easy}
Let $L$ be an SPS lattice with the covering square~$S$.
Then $L[S]$ is an SPS lattice and $L$ is a sublattice. 
Therefore, every element $x$ of $L[S]$ 
has an upper cover $x^+$ and a lower cover $x^-$ in $L$.
\end{lemma}

We name a few join-irreducible congruences of $L$ and $L[S]$ 
that will play an important role. 

Join-irreducible congruences in $L$:
\begin{align}
\bga_l(S) &= \consub{L}{a_l,t},\\
\bga_r(S) &= \consub{L}{a_r,t}. 
\end{align}

Join-irreducible congruences in  $L[S]$, see Figure~\ref{F:forkdetails}:
\begin{align}
\ol\bga_l(S) &= \consub{L[S]}{a_l,t},\\ 
\ol\bga_r(S) &= \consub{L[S]}{a_r,t},\label{F:alphabarright}\\ 
\bgg(S) &= \consub{[L[S]}{m,t}. 
\end{align}

\section{Extending congruences}\label{S:Extending}

In this section we prove Theorem~\ref{T:extension}.

\begin{proof}[Proof of \emph{Theorem~\ref{T:extension}\lp i\rp}]

Let $\bga$ be a congruence of~$L$ satisfying $\bga \restr S = \one_S$.
We define the partition:
\[
  \bgp =  \setm{[u,v]_{L[S]}}{\text{$u, v \in L$ and }[u,v]_L 
  \text{\ is a congruence class of $\bga$}}.
\]
To verify that $\bgp$ is indeed a partition of $L[S]$, let
\[
   A =  \UUm{[u,v]_{L[S]}}{\text{$u, v \in L$ and }[u,v]_L 
        \text{\ is a congruence class of $\bga$}}.
\]
Clearly, $L \ci A$. 
By assumption, $[o, t]$ is in a
congruence class of $\bga$, so there is a 
congruence class $[u,v]_{L}$ containing $o$ and $t$. 
Hence  $[u,v]_{L[S]} \in \bgp$. 
Since $\cng o=a_l (\bga)$, 
and so $\cng x_{l_i}= y_{l_i} (\bga)$ for $i = 1, \dots, n_l$,
therefore, there is a congruence class $[u_i,v_i]_{L}$ 
containing $x_{l,i}$ and $ y_{l,i}$ for $i = 1, \dots, n_l$.
Hence $z_{l,i} \in [u_i,v_i]_{L[S]} \ci A$ 
and symmetrically. This proves that $A = L$.

Next we observe that $x$ belongs to a $\bgp$-class if{f} so do $x^+$ and $x^-$. This implies that the sets in $\bgp$ are pairwise disjoint.

Finally, we verify the substitution properties. 
Let $a,b,c \in L[S]$ and $\cng a=b(\bgp)$.
Then there exist $u,v \in L$ with  $\cng u=v(\bga)$ such that  $a,b \in [u,v]_{L[S]}$. 
There is also an interval $[u',v']_{L[S]} \in \bgp$ with $c \in [u',v']_{L[S]}$.
Since $\cng u=v(\bga)$ and  $\cng u'=v'(\bga)$, it follows that 
$\cng u\jj u'=v \jj v'(\bga)$, so there is a congruence class $[u'',v'']$ of $\bga$ 
containing $u\jj u'$ and $v \jj v'$. 
So $u\jj u', v \jj v' \in [u'',v'']_{L[S]}$, 
verifying the substitution property for joins. 
The dual proof verifies the substitution property for meets.

So $\bgp$ is a congruence of $L[S]$. Clearly, $\bgp = \ol \bga$.
The uniqueness statement is obvious.
\end{proof}

Curiously, Lemma~\ref{L:technical} would not simplify this proof.

\begin{proof}[Proof of \emph{Theorem~\ref{T:extension}\lp ii\rp}]
Now let $\bga$ be a congruence of~$L$ 
satisfying $\bga \restr S = \zero_S$.
We are going to define a  congruence $\bgb$ 
of $L[S]$ extending $\bga$.

For $i = 1, \dots, n_l$, define $\ul i$ 
as the smallest element in $\set{1, \dots, n_l}$
satisfying $\cng x_{l,i} = x_{l,\ul i}(\bga)$; 
let $\ol i$ be the largest one. Clearly, $\ul i \leq i \leq \ol i$.
Similarly, by a slight abuse of notation, 
we define $\ol i$ and $\ul i$ on the right.
 
We define $\bgb$ as the partition:
\begin{align}\label{E:zero}
  &\setm{[u,v]_{L[S]}}{[u,v]_L 
  \text{\ is a congruence class of $\bga$ with $u < v$}}\\
 & \uu \set{m}
  \uu \setm{[z_{l,\ul i}, z_{l,\ol i}]}{i=1, \dots, n_l}
   \uu \setm{[z_{r,\ul i}, z_{r,\ol i}]}{i=1, \dots, n_r}.\notag
\end{align}
To see that $\bgb$ is a partition, observe that it covers $L[S]$.
Note that if $u,v \in L$ with $u < v$ such that  $[u,v]_L$ 
is a congruence class of $\bga$,
then $[u,v]_L = [u,v]_{L[S]}$ unless $u < z_{l,i} < v$ 
for some $i=1, \dots, n_l$ or symmetrically. 
But in this case we would have that 
$\cng y_{l,i} = x_{l,i} (\bga)$
implying that $o = \cng y_{l,1} = x_{l,1} = a_l (\bga)$,
contrary to the assumption. Clearly, two distinct
$[z_{l,\ul i}, z_{l,\ol i}]$ classes cannot intersect 
by the definition of $\ul i$ and~$\ol i$.

So if two classes intersect, it must be a $[u,v]_{L[S]}$
and a $[z_{l,\ul i}, z_{l,\ol i}]$, which would contradict that
$[u,v]_{L[S]} = [u,v]_{L}$. 
The symmetric (on the right) 
and the mixed cases (left and right) complete the discussion.
So \eqref{E:zero} defines a partition $\bgb$.

To see that $\bgb$ is a congruence, we use Lemma~\ref{L:technical}.
To verify (C${}_\jj$), 
let $a \prec b$, $a \prec c \in L[S]$, $b \neq c$, 
and $\cng a = b (\bgb)$. We want to prove that $\cng c = b \jj c (\bgb)$.
By \eqref{E:zero}, either $\cng a=b(\bga)$ 
or $a, b \in [z_{l,\ul i}, z_{l,\ol i}]$ for some $i=1, \dots, n_l$
(or symmetrically).

First, let $a, b \in L$ and $\cng a=b(\bga)$. 
If $c \in L$, then $\cng c = b \jj c (\bga)$
since $\bga$ is a congruence, so by \eqref{E:zero}, $\cng c = b \jj c (\bgb)$.
 
If $c \nin L$, that is, if $c \in F[S]$,
then $c \neq m$ or $c = m$.
Clearly, $c \neq m$ cannot happen, 
since $m$ does not cover an element not in $F[S]$.
Therefore, $c = m$ and so $c = z_{l, i}$ 
for some $i=1, \dots, n_l$ (or symmetrically).
Since $a \prec c = z_{l, i}$ and $a \in L$, 
we get that $a = y_{l, i}$, $i > 1$ and $b = y_{l, i-1}$. 
It easily follows that $\cng c=c\jj b(\bga)$.

Second, let $a,b \in [z_{l,\ul i}, z_{l,\ol i}]$ for some $i=1, \dots, n_l$
(or symmetrically). Then $a = z_{l, j}$ and $b = z_{l, j-1}$
with $\cng z_{l, j}=z_{l, j-1}(\bga)$ by \eqref{E:zero}, 
and $c = x_{l, j}$. So $\cng c = b \jj c (\bgb)$ by \eqref{E:zero},
completing the verification of (C${}_\jj$).

To verify (C${}_\mm$), let $a \succ b$, $a \succ c \in L[S]$, $b \neq c$, 
and $\cng a = b (\bgb)$. We want to prove that $\cng c = b \mm c (\bgb)$.
By \eqref{E:zero}, either $a, b \in L$ and $\cng a=b(\bga)$ 
or $a, b \in [z_{l,\ul i}, z_{l,\ol i}]$ for some $i=1, \dots, n_l$
(or symmetrically).

First, let $a, b \in L$ and $\cng a=b(\bga)$. 
If $c \in L$, then $\cng c = b \mm c (\bga)$
since $\bga$ is a congruence, so by \eqref{E:zero}, 
$\cng c = b \mm c (\bgb)$. If $c \nin L$, that is, if $c \in F[S]$,
then $c \neq m$; indeed, if $c = m$, then $a = t$ and $b = a_l$, or symmetrically. 
But then $\cng a=b(\bga)$ would contradict the assumption that
$\bga \restr S = \zero_S$.

So $c = z_{l, i}$ for some $i=1, \dots, n_l$ (or symmetrically).
Since $a \succ c = z_{l, i}$ and $a \in L$, 
we get that $a = x_{l, i}$, $i > 1$. Now if $b = x_{l, i+1}$, 
then $\cng c=b\mm c(\bgb)$ easily follows from \eqref{E:zero}. 
So we can assume that $b \neq x_{l, i+1}$. 
Then by Lemma~\ref{L:known}(ii), the elements $b$, $x_{l, i+1}$,
and $c$ generate an $\SN 7$ sublattice; therefore,
$\cng c=b\mm c(\bgb)$ is easily computed in the $\SN 7$ sublattice. 

Second, let $a,b \in [z_{l,\ul i}, z_{l,\ol i}]$ for some $i=1, \dots, n_l$
(or symmetrically). Then $a = z_{l, j}$, $b = z_{l, j+1}$
and $c = y_{l, j}$. By the definition of $\ul i$ and $\ol i$, 
it follows that $\cng z_{l, j}= z_{l, j+1} (\bga)$,
which trivially implies that $\cng y_{l, j}= y_{l, j+1} (\bga)$,
that is, $\cng c= b \mm c (\bga)$, and so $\cng c= b \mm c (\bgb)$.

Clearly, $\bgb = \oa$.
\end{proof}

Note the similarity between the proofs of the conditions
(C${}_\jj$) and (C${}_\mm$).
Unfortunately, there is no duality.

\begin{proof}[Proof of \emph{Theorem~\ref{T:extension}\lp iii\rp}]
We need two examples. 
The first is trivial: 
Let $L = S = \SC 2^2$. Then $L[S] =\SfS 7$ 
and all congruences of $L$ extend to $L[S]$.

For the second, see Figure~\ref{F:notextend}. 
Let $\bga$ let be the congruence of $L$ collapsing two opposite sides of the covering square $S$. This congruence has exactly two nontrivial classes, marked in Figure~\ref{F:notextend} by bold lines.
In particular, $\cng a=b(\bga)$ fails.

Figure~\ref{F:notextend} also shows the congruence
$\ol \bga$ of $L[S]$.
Note that $\cng a=b(\oa)$ in $L[S]$, so $\bga$ has no extension to $L[S]$.
\end{proof}

\begin{figure}[t]
\centerline{\includegraphics{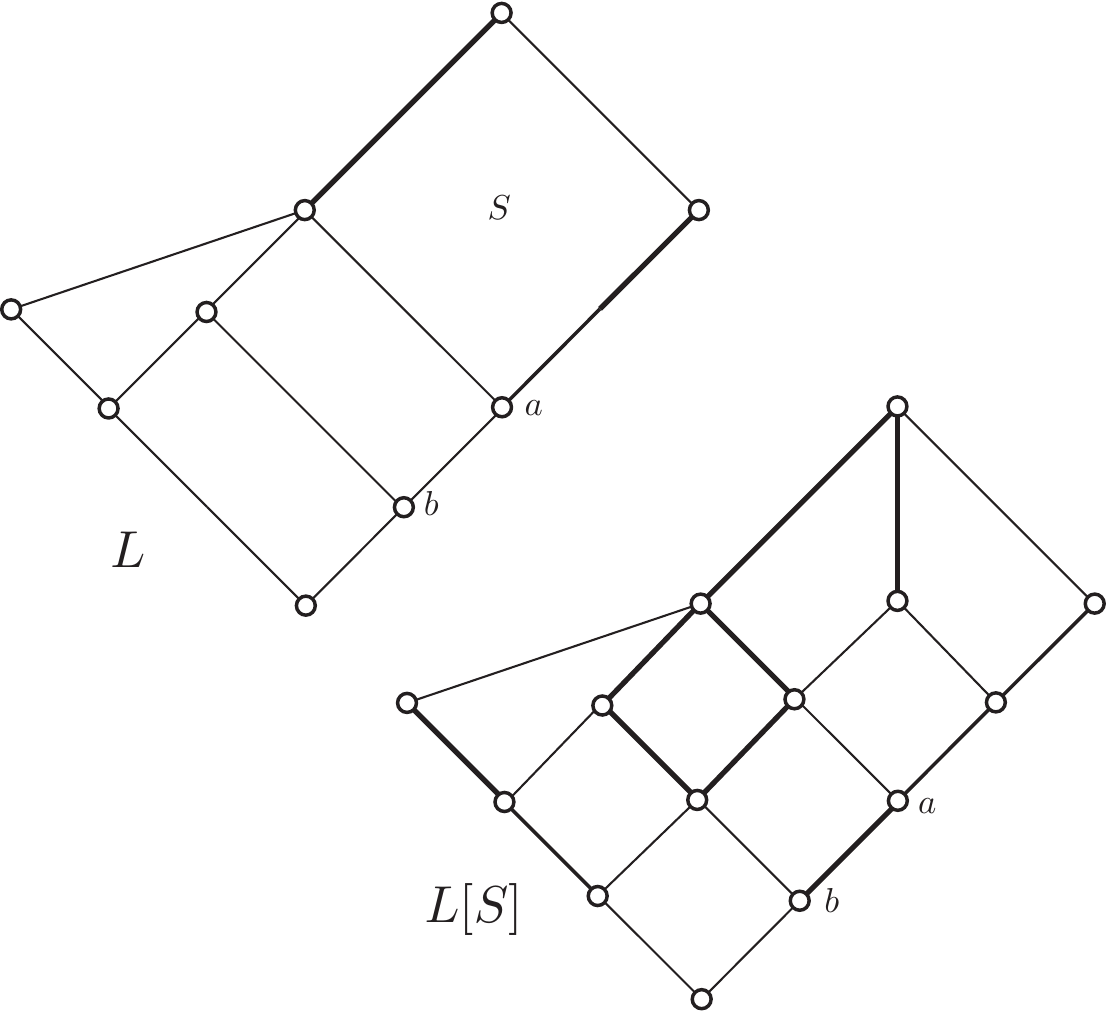}}
\caption{A congruence of $L$ that does not extend to $L[S]$}\label{F:notextend}
\end{figure}

Note that $\con{a, b}$ is the protrusion congruence $\bgp$
as defined in Section~\ref{S:gammaK}.

\section{Proof of Theorem~\ref{T:wide}}\label{S:wide}
In this section, we prove Theorem~\ref{T:wide}.

Since $S$ is wide, 
the element $t$ covers an element $a$ in $L$, with $a \neq a_l, a_r$.
Since $L$ is slim, 
either $a$ is to the left of $a_l$ or to the right of $a_r$;
let us assume the latter, see Figure~\ref{F:widesquare}.
By~Lemma~\ref{L:known}, the set $\set{a_l,a_r,a}$ 
generates an $\SfS 7$ sublattice in $L$.

Then $\con{m,t} \leq \con{a_r,t}$, 
computed in the $\SfS 7$ sublattice 
generated by the set $\set{a_l,m,a_r}$,
and $\con{m,t} \geq \con{a_r,t}$, 
computed in the $\SfS 7$ sublattice generated by $\set{m,a_r,a}$,
so we conclude that $\bgg(S) = \oa_r(S)$.

So by \eqref{F:alphabarright}, $\bgg(S)$ is generated by a congruence of $L$, namely by $\con{t,a}$.

\begin{figure}[p]
  \centerline{\includegraphics{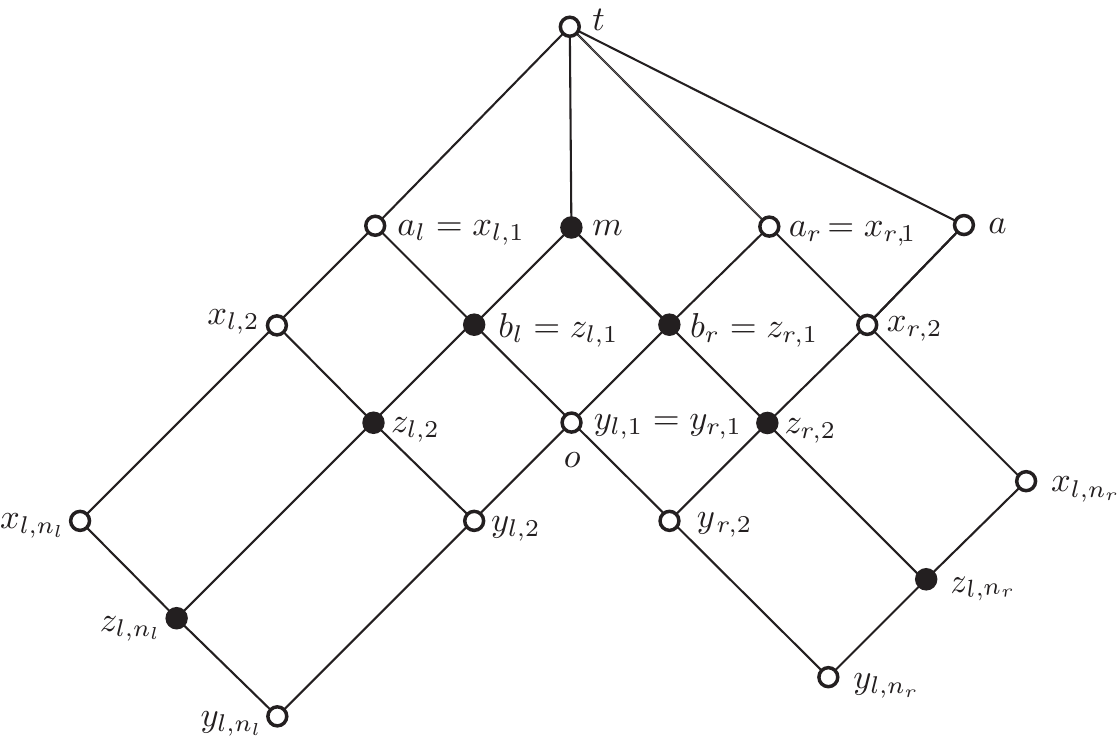}} 
  \caption{Wide square}\label{F:widesquare}
  
\bigskip

\bigskip

  \centerline{\includegraphics[scale=0.85]{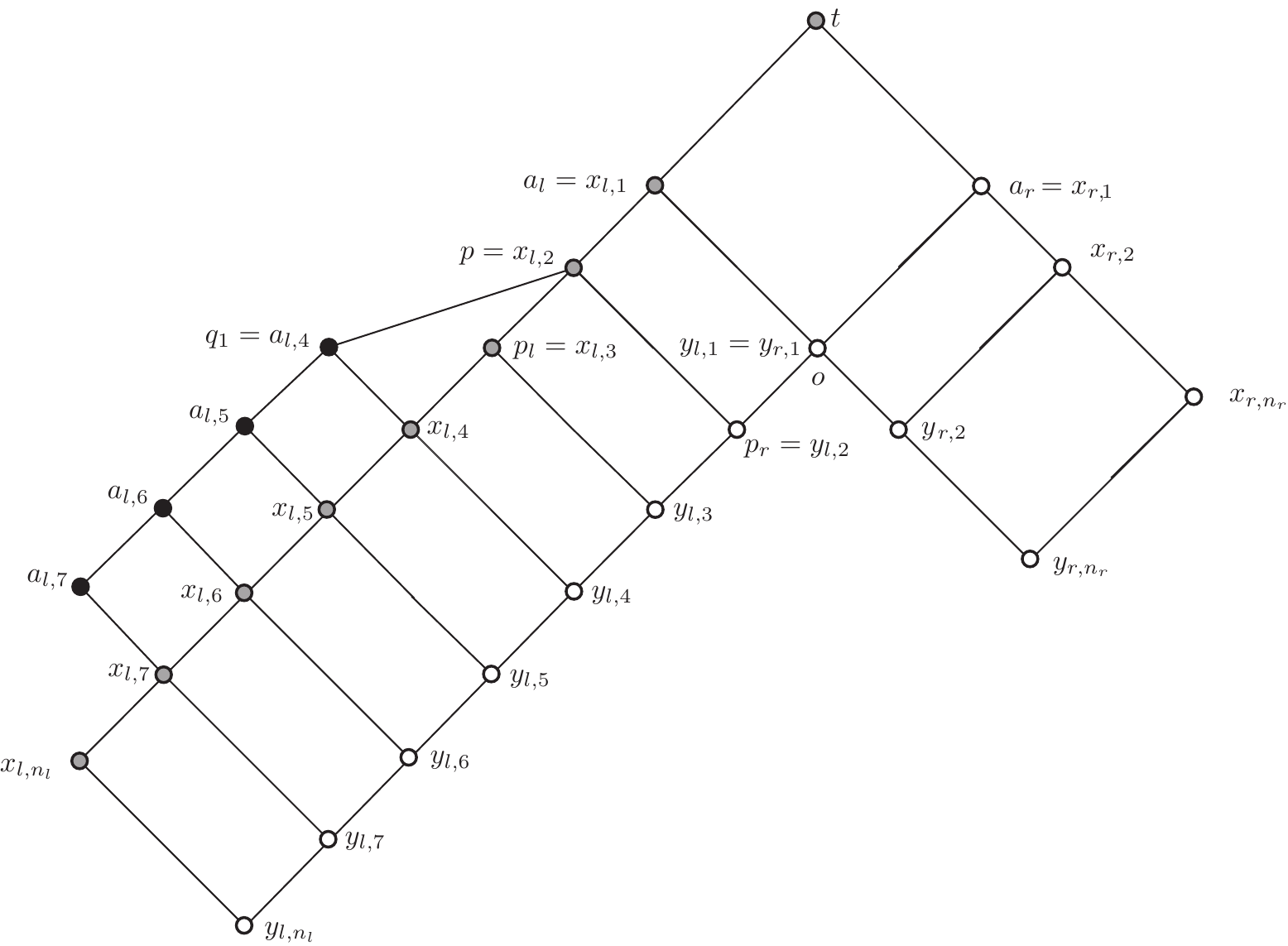}}
  \caption{The lattice $L$, the maximal chain $C$, gray filled, 
  a~protrusion at $p = x_{l,2}$ (with $p_l = x_{l,3}$, 
  $p_r = y_{l,2}$), and the protrusion $\Prot(p)$, black filled}
  \label{F:protrusion}
\end{figure}

\section{Protrusions}\label{S:Protrusion}
Let $L$ be an SPS lattice 
and let $S = \set{o, a_l, a_r, t}$ be a tight covering square of~$L$.
Let $C_l$ be the maximal chain 
$x_{l,n_l} \prec \dots \prec x_{l,1} \prec t$ in $[x_{l,n_l}, t]$,
the gray filled elements of Figure~\ref{F:protrusion}.
Recall that $x_{l,n_l}$, and $y_{l,n_l}$ are on the left boundary of $L$.
Let the chain $\ol C_l$ extend $C_l$ 
by continuing the chain on the left boundary of $L$.
So $\ol C_l$ extends to the zero of $L$.
We define $C_r$ and $\ol C_r$ symmetrically. 

Now we define a \emph{protrusion on $C_l$} 
or \emph{a left-protrusion of $S$}---and symmetrically. 
(Observe that $S$ determines $C_l$.)
If no element $x < t$ of $C_l$ covers an element to the left of $C_l$,
then there is no protrusion on $C_l$. 
If there is such an element $x < t$, pick the largest one, $p$.
Note that $p$ covers exactly two elements $p_l$ and $p_r$ 
on or to the right of $C_l$, $p_l$ to the left of $p_r$.
Now let $q_1$ be the element covered by $p$ 
to the left of $C_l$ and immediately to the left of $q_l$.
See Figure~\ref{F:protrusion}.

Let $p = x_{l,k}$. Then $p_l = x_{l,k+1}$ and $p_r = y_{l,k}$.

\begin{lemma}\label{L:nice!}
$q_1 \mm p_{l} = x_{l,k+2}$ and $x_{l,k+2} \prec p_l$.
\end{lemma}

\begin{proof}
Since $q_1$ is immediately to the left of $p_l$, 
it follows that $q_1$, $p_l$, and $p_r$ are adjacent.
By Lemma~\ref{L:known}(iii), they generate a cover-preserving $\SfS 7$ sublattice of $L$. This sublattice is 
\[
   \set{p, q_1, p_l, p_r,
      q_1 \mm p_l, p_l \mm p_r,q_1 \mm p_r}.
\]
Since this is a cover-preserving sublattice and 
$q_1 \mm p_r < x_{l,k+2} \leq q_1 \mm p_l$,
it follows that $q_1 \mm p_l = x_{l,k+2}$, as claimed.
The second statement is obvious now because the elements
$x_{l,k+2} < q_1$ are in a cover-preserving sublattice.
\end{proof}

Now we take the left wing of $[x_{l,k+2}, q_1]$:
\[
   [x_{l,k+2},a_{l,k+2}], \dots, [x_{l,k^*}, a_{l,k^*}],
\]
where $q_1 = a_{l,k+2}$ and $k \leq k^* \leq n_l$. 
Note that $[x_{l,k^*}, a_{l,k^*}]$ is on the left boundary of $L$.

The set $\New(p) = \set{a_{l,k+2}, \dots, a_{l,k^*}}$ 
forms the \emph{protrusion at $p$}; 
we may also call the element $p$ a protrusion.  
There is always one element in the protrusion, namely, $q_1 = a_{l,k+2}$. 
In Figure~\ref{F:protrusion}, 
$k = 2$ and $k^* = 7$; there are four elements 
in the protrusion, black filled.

%

\section{Covering squares with no protrusions}\label{S:noprotrusions}

Let $L$ be an SPS lattice and let $S$ be a covering square with no protrusion.
We define in $L[S]$ an equivalence relation $\bgd$ as follows.
All equivalence classes of $\bgd$ are singletons 
except for the following intervals:
\begin{align}\label{E:seti}
[z_{l,1}, x_{l,1}],& \dots, [z_{l,n_l}, x_{l,n_l}],\\ 
&[m, t], \label{E:setii}\\
[z_{r,1}, x_{r,1}],& \dots, [z_{r,m_r}, x_{r,m_r}].\label{E:setiii}
\end{align}

\begin{lemma}\label{L:distr}
Let $L$ be an SPS lattice and let $S$ be a covering square with no protrusion. 
Then $\bgd$ is a congruence relation of $L[S]$, 
in fact, $\bgd = \bgg(S)$. 
\end{lemma}

\begin{proof}
By Lemma~\ref{L:technical}, 
we have to verify (C${}_{\jj}$) and (C${}_{\mm}$).
To verify (C${}_{\jj}$), let 
$x$ be covered by $y \neq z$ in $L[S]$ 
and let $\cng x=y(\bgd)$. If $x = z_{l,i}$ and $y = x_{l,i}$, 
where $1 < i \leq n_l$, then $z = z_{l,i-1}$ and
$\cng {z= x_{l,i-1}}= {y \jj z = x_{l,i-1}}(\bgd)$, 
because $\set{z_{l,i-1}, x_{l,i-1}}$ is in the list \eqref{E:seti}.
If $i = 1$, we proceed the same way with $z = m$.
We proceed ``on the right'' 
with the lists \eqref{E:setii} and \eqref{E:setiii}.
Finally, $x = m$ cannot happen because $m$ is covered only by one element.

To verify (C${}_{\mm}$), let $x$ cover $y \neq z$ in $L[S]$ 
and let $\cng x=y(\bgd)$.
The verification of (C${}_{\mm}$) is very similar 
to the arguments is the previous paragraph, 
except for $x = x_{l,i}$ and $y = z_{l,i}$, 
where $1 < i \leq n_l$, we have to verify that $z = z_{l,i-1}$; indeed, 
if $z \neq z_{l,i-1}$ then $z$ defines a protrusion, contradicting the assumption.
\end{proof}          
 
Let $L$ be an SPS lattice 
and let $S$ be a covering square. 
Let us call $S$ \emph{distributive} 
if the ideal generated by $S$ is distributive.
A distributive covering square has no protrusion. 
So Lemma \ref{L:distr} is closely related 
to a construction in G. Cz\'edli~\cite[p. 339]{gC12}, 
where a~fork is inserted into a distributive covering square.

\section{The congruence $\bgg$ on $K$}\label{S:gammaK} 
Let $L$, $S = \set{o, a_l, a_r, t}$, $C_l$, and $C_r$ 
be as in Section~\ref{S:Protrusion} and let $S$ be tight.  
We assume that there is a~protrusion, $p = x_{l,k}$ on~$C_l$.

\begin{figure}[b]
\centerline{\includegraphics[scale=0.85]{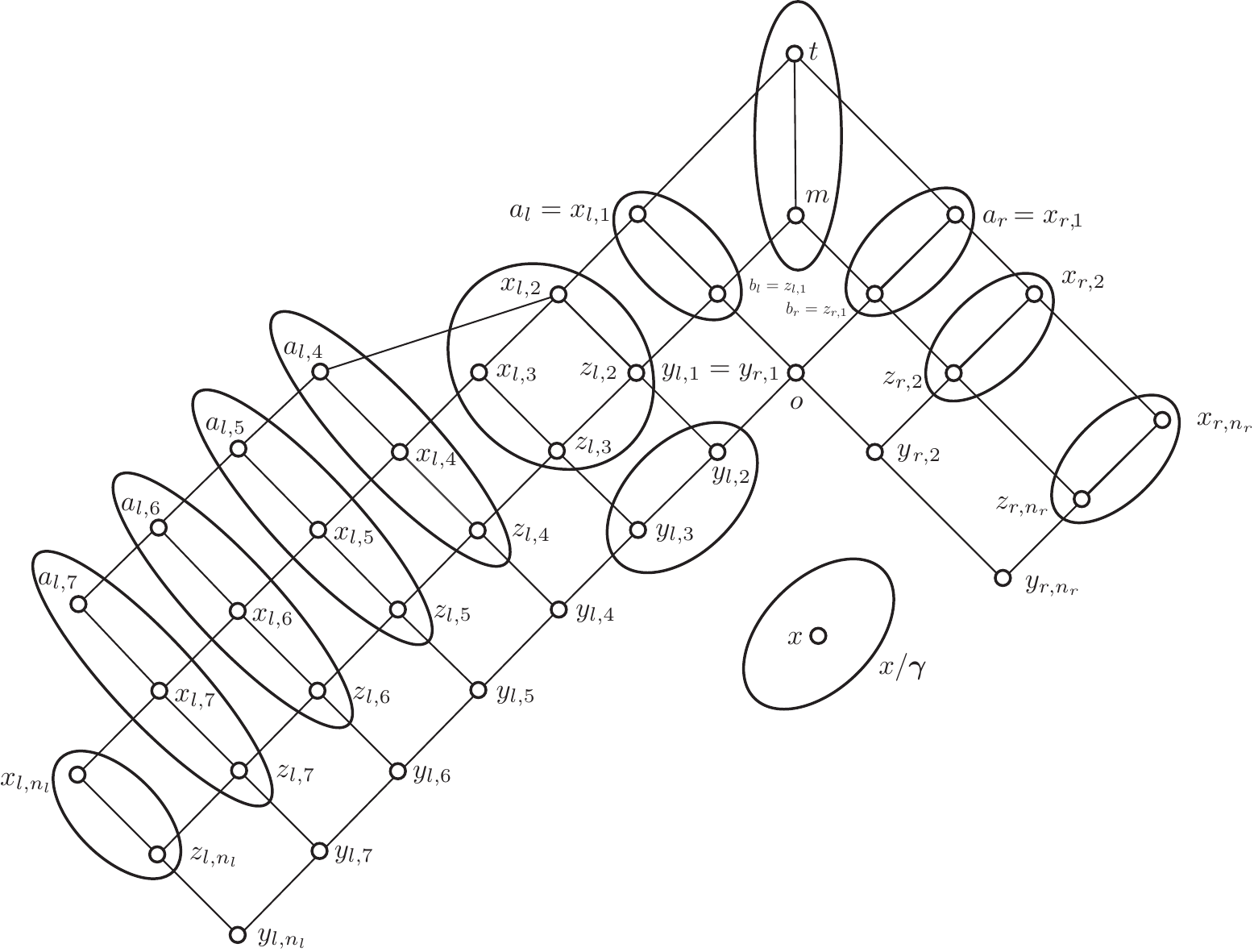}}
\caption{The congruence $\bgg$ on $K$}\label{F:gamma} 
\end{figure}

Let $\Btw$ be the set of elements of $L[S]$ \emph{between
$\ol C_l$ and $\ol C_r$} (including $\ol C_l$ and~$\ol C_r$).
Let 
\[
   K = \Btw \uu \New(p).
\]

\begin{lemma}\label{L:K}
$K$ is a lattice. It is a sublattice of $L[S]$.
\end{lemma}

\begin{proof}
Both $\Btw$ and $\New(p)$ are sublattices of $L[S]$.
If $x \in \Btw$ and $y \in \New(p)$, then $x \mm y \in \Btw$. 
For $x \in \Btw$, let $x^+$ denote the smallest element of $C_l$
containing $x$. It is easy to see that $x \jj y = x^+$ if $x \nleq y$.
The statements follow.
\end{proof}

Now we define an equivalence relation $\bgg$ on $K$ 
($\bgg$ will be a part of $\bgg(S)$):

Let 
\begin{equation}
   t/\bgg = \set{m, t}.\label{E:i}
   \end{equation}

If $i < k$ or $k^* < i$, then
\begin{equation}
   x_{l,i}/\bgg = \set{x_{l,i}, z_{l,i}}.\label{E:ii}
\end{equation}

Let
\begin{align}
   x_{l,k}/\bgg 
   &= \set{x_{l,k}, x_{l,k+1},z_{l,k},z_{l,k+1}},\label{E:iii}\\
   y_{l,k}/\bgg &= \set{y_{l,k}, y_{l,k+1}}.\label{E:iv}\\
   \intertext{If $x_{l,i} \in \New(p)$, 
   that is, $k+2 \leq i \leq i^*$, then}
    x_{l,i}/\bgg 
   &= \set{a_{l,i}, x_{l,i}, z_{l,i}}.\label{E:v}
\end{align}

We define $x_{r,i}/\bgg$ symmetrically.

Define $\bgp = \consub{L}{y_{l,k}, y_{l,k+1}}$, a congruence of $L$,
and $\bgp_K = \consub{K}{y_{l,k}, y_{l,k+1}}$, a~congruence of $K$. 

Let $x/\bgg$ be as defined above in \eqref{E:i}--\eqref{E:v}; let  
\begin{equation}\label{E:vi}
   x/\bgg = x/\bgp_K,
\end{equation}
otherwise, see Figure~\ref{F:gamma}.
Clearly, $\bgp_K \leq \bgg$.

\begin{lemma}\label{L:gamma}
Let $L$ be an SPS lattice. 
Let $S$ be a \emph{tight square} and $p = x_{l,k}$ be a~protrusion on~$C_l$.
Then the binary relation $\bgg$ 
defined by \eqref{E:i}--\eqref{E:vi}
is a congruence relation on $K$.
\end{lemma}

\begin{proof}
Since $\bgg$ is an equivalence relation 
with intervals as equivalence classes,
by Lemma \ref{L:technical}, 
we only have to verify (C${}_{\jj}$) and (C${}_{\mm}$).

\emph{To verify} (C${}_{\jj}$),
let $v \neq w$ cover $u$ and let $\cng u = v (\bgg)$.
Then we distinguish six cases according to \eqref{E:i}--\eqref{E:vi}
in the definition of $\bgg$.

\emph{Case} (C${}_{\jj}$,\ref{E:i}): $u = m$. This cannot happen 
because $m$ has only one cover.

\emph{Case} (C${}_{\jj}$,\ref{E:ii}): $u = z_{l,i}$, $v = x_{l,i}$. 
Using the notation $m = z_{l,0}$ and $t = x_{l,0}$, 
then $w = z_{l,i-1}$ and so 
\[
   \cng {w = z_{l,i-1}} = {x_{l,i-1} = v \jj w} (\bgg)
\]
by (C${}_{\jj}$, \ref{E:ii}) or (C${}_{\jj}$, \ref{E:iii}).  

\emph{Case}(C${}_{\jj}$,\ref{E:iii}): $u,v \in \set{x_{l,k}, x_{l,k+1},z_{l,k},z_{l,k+1}}$.
 If $u = z_{l,k}$, $v = x_{l,k}$
or if $u = z_{l,k+1}$, $v = x_{l,k+1}$, 
we proceed as in Case (C${}_{\jj}$, \ref{E:ii}); the conclusion holds 
by \eqref{E:i} or \eqref{E:ii}. 
$u = x_{l,k+1}$ cannot happen because $x_{l,k+1}$ 
has only one cover in $K$. 
So we are left with $u = z_{l,k+1}$, $v = z_{l,k}$. 
Then $w = x_{l,k+1}$ and 
\[
   \cng {w = x_{l,k+1}} = {x_{l,k} = v \jj w } (\bgg).
\]

\emph{Case} (C${}_{\jj}$,\ref{E:iv}): $u,v \in \set{y_{l,k}, y_{l,k+1}}$. 
Then $u = y_{l,k+1}$, $v = y_{l,k}$, and $w = x_{l,k+1}$. Therefore, 
\[
   \cng {w = z_{l,k+1}} = {z_{l,k} = v \jj w} (\bgg)
\]
by \eqref{E:iii}.

\emph{Case} (C${}_{\jj}$,\ref{E:v}): $u,v \in \set{a_{l,k+2}, x_{l,k+2}, z_{l,k+2}}$.
In this case, either $u = z_{l,k+2}$, $v = x_{l,k+2}$
or $u = x_{l,k+2}$, $v = a_{l,k+2}$. 
If $u = z_{l,k+2}$, $v = x_{l,k+2}$, then $w = z_{l,k+1}$ and 
\[
   \cng {w = z_{l,k+1}} = {x_{l,k+1} = v \jj w} (\bgg)
\]
by \eqref{E:iii}. 
If $u = x_{l,k+2}$, $v = a_{l,k+2}$, then $w = x_{l,k+1}$ and
\[
   \cng {w = x_{l,k+1}} = {x_{l,k} = v \jj w } (\bgg)
\]
by \eqref{E:iii}.

\emph{Case} (C${}_{\jj}$,\ref{E:vi}): $u, v \in x/\bgp$ 
and so that \eqref{E:i}--\eqref{E:vi} do not apply. 
Then using the notation $o = y_{l,0} = y_{r,0}$, 
we have that $u <  y_{l,i}$ for some $i=0, \dots, n_l$
(or symmetrically).
Therefore, $\cng w = v \jj w (\bgp)$ since $\bgp$ is a congruence on~$K$.

\emph{To verify} (C${}_{\mm}$), let $u$ cover $v \neq w$ 
and let $\cng v = u (\bgg)$. Then we again distinguish six cases.

\emph{Case} (C${}_{\mm}$,\ref{E:i}): $u = t$, $v = m$. 
Since $S$ is tight, $w = x_{l,1}$ or symmetrically. 
Then 
\[
   \cng {w = x_{l,1}} = {z_{l,1} = v \mm w} (\bgg)
\]
by \eqref{E:ii}.

\emph{Case} (C${}_{\mm}$,\ref{E:ii}): $u = x_{l,i}$, $v = z_{l,i}$. 
Then $w = x_{l,i+1}$ and so 
\[
   \cng {w = x_{l,i+1}} = {v \mm w = z_{l,i+1}} (\bgg)
\]
by \eqref{E:ii}.

\emph{Case} (C${}_{\mm}$,\ref{E:iii}): 
$u,v \in \set{x_{l,k}, x_{l,k+1},z_{l,k},z_{l,k+1}}$.

There are four subcases:

\emph{Subcase 1}: $u = x_{l,k}$, $v = z_{l,k}$.

\emph{Subcase 2}: $u = x_{l,k+1}$, $v = z_{l,k+1}$.

\emph{Subcase 3}: $u = x_{l,k}$, $v = x_{l,k+1}$.

\emph{Subcase 4}: $u = z_{l,k}$, $v = z_{l,k+1}$.

For subcases 1 and 2, argue as in Case (C${}_{\mm}$,\ref{E:ii}).

For subcase 3, $w = z_{l,k}$ or $w = a_{l,k+2}$ in $K$. 
If $w = z_{l,k}$, then 
\[
   \cng {w = z_{l,k}} = {z_{l,k+1} = v \mm w} (\bgg)
\]
by \eqref{E:iii}. If $w = a_{l,k+2}$, then $w = v \mm w$, 
therefore, $\cng w = v \mm w (\bgg)$.

For subcase 4, $w = y_{l,k}$, so  
$\cng {w = y_{l,k}} = {y_{l,k+1} = v \mm w} (\bgg)$ 
by~\eqref{E:vi}.

\emph{Case} (C${}_{\mm}$,\ref{E:iv}): 
$u = y_{l,k}$, $v = y_{l,k+1}$. 
Then $\cng u = v (\bgp)$ by \eqref{E:vi}, so  
\[
   \cng w = {v \mm w} (\bgp)
\]
implying that
\[
   \cng w = {v \mm w} (\bgg)
\]
by \eqref{E:vi}.

\emph{Case} (C${}_{\mm}$,\ref{E:v}):  
$k+2 \leq i \leq i^*$ and $u,v \in \set{a_{l,i}, x_{l,i}, z_{l,i}}$. 
If $u = x_{l,i}$, $v = z_{l,i}$, 
we proceed as in Case (C${}_{\mm}$, \ref{E:ii}).
Otherwise, $u = a_{l,i}$, $v = x_{l,i}$. Then $w \prec a_{l,i}$ in $K$.
This cannot happen if $i = k^*$, since $a_{l,k^*}$ 
is on the boundary of $L$ and therefore of $K$. So $i < k^*$.
Then $a_{l,i+1}$ (and only $a_{l,i+1}$) is covered by $a_{l,i}$
but is distinct from $x_{l,i}$ in $K$ 
(there may be other elements of $L$ covered by $a_{l,i}$, 
but they are not in $K$). Therefore, 
\[
   \cng {w = a_{l,i+1}} = {x_{l,i+1} = v \mm w} (\bgg)
\]
by \eqref{E:v}.

\emph{Case} (C${}_{\mm}$,\ref{E:vi}): $u, v \in x/\bgp_K$. 
Then $\cng w = v \mm w (\bgp_K)$ since $\bgp_K$ is a congruence on~$K$.
So $\cng w = v \mm w (\bgg)$ by \eqref{E:vi}.
\end{proof}

We call $\bgp$ a \emph{protrusion congruence on~$L$}
and $\ol\bgp$ a \emph{protrusion congruence on~$L[S]$}; 
see Figure \ref{F:notextend} for the simplest example.

\section{The congruence $\bgg(S)$ on $L$}\label{S:gammaL}
Let $L$, $S = \set{o, a_l, a_r, t}$, $C_l = C_{l, 1}$, and $C_r= C_{r, 1}$ 
be as in Section~\ref{S:Protrusion}; 
recall that we assume that $S$ is tight. 

We want to describe the congruence 
$\bgg(S) = \consub{L[S]}{m,t}$ on $L[S]$.

By Lemma~\ref{L:distr},
if there is no protrusion on $C_l$ or on $C_r$, 
then $\consub{L[S]}{m,t}$ is the
congruence~$\bgd$ described in \eqref{E:seti}--\eqref{E:setiii}
by stipulating that all the other congruence classes are trivial.
The protrusion congruence $\bgp$ is trivial ($\bgp = \zero$).

So now we can assume that there is a~protrusion. 
If there is more than protrusion on~$C_l$, 
we choose the largest $p_1 = x_{l,k_1}$.

Let $K_1 = K = \Btw \uu \New(p)$.
Let $\bgg_1 = \bgg$ and $\bgp_1 = \bgp$ 
be the congruence relation on $K_1$ 
as defined in Section~\ref{S:gammaK}.

In $C_{l, 1}$, let us replace $\set{x_{l,k_1+1}, \dots, x_{l,k_1^*}}$ by 
$\set{a_{l,k_1+2}, \dots, a_{l,k_1^*}}$. 
We obtain the maximal chain $C_{l,2}$ of $K_1$.
Note that $C_{l,2}$ is shorter than $C_{l,1}$.

If there is no protrusion on $C_{l,1}$, we can extend $\bgg_1$ to $L[S]$
by using the congruence $\bgp = \con{y_{l,i}, y_{l,i+1}}$ of~$L$, 
and letting  
\begin{equation}\label{E:ext}
   x/\bgg_1 = x/\bgp,
\end{equation}
for $x \in L - K_1$.

Figure \ref{F:protrusionplus} illustrates this process. 
The top diagram shows the case when $p_1$ covers four elements, 
so $p_1$ remains the protrusion in the second step. 
The element $p_1$ and the four elements it covers generate a multifork
as defined in G. Cz\'edli~\cite{gC13}.
In the middle diagram there is a new protrusion $p_2 \in \Prot(p_1)$. 
In the bottom diagram there is a new protrusion $p_2 \nin \Prot(p_1)$. 

\begin{figure}[h!]
\centerline{\includegraphics[scale=0.55]{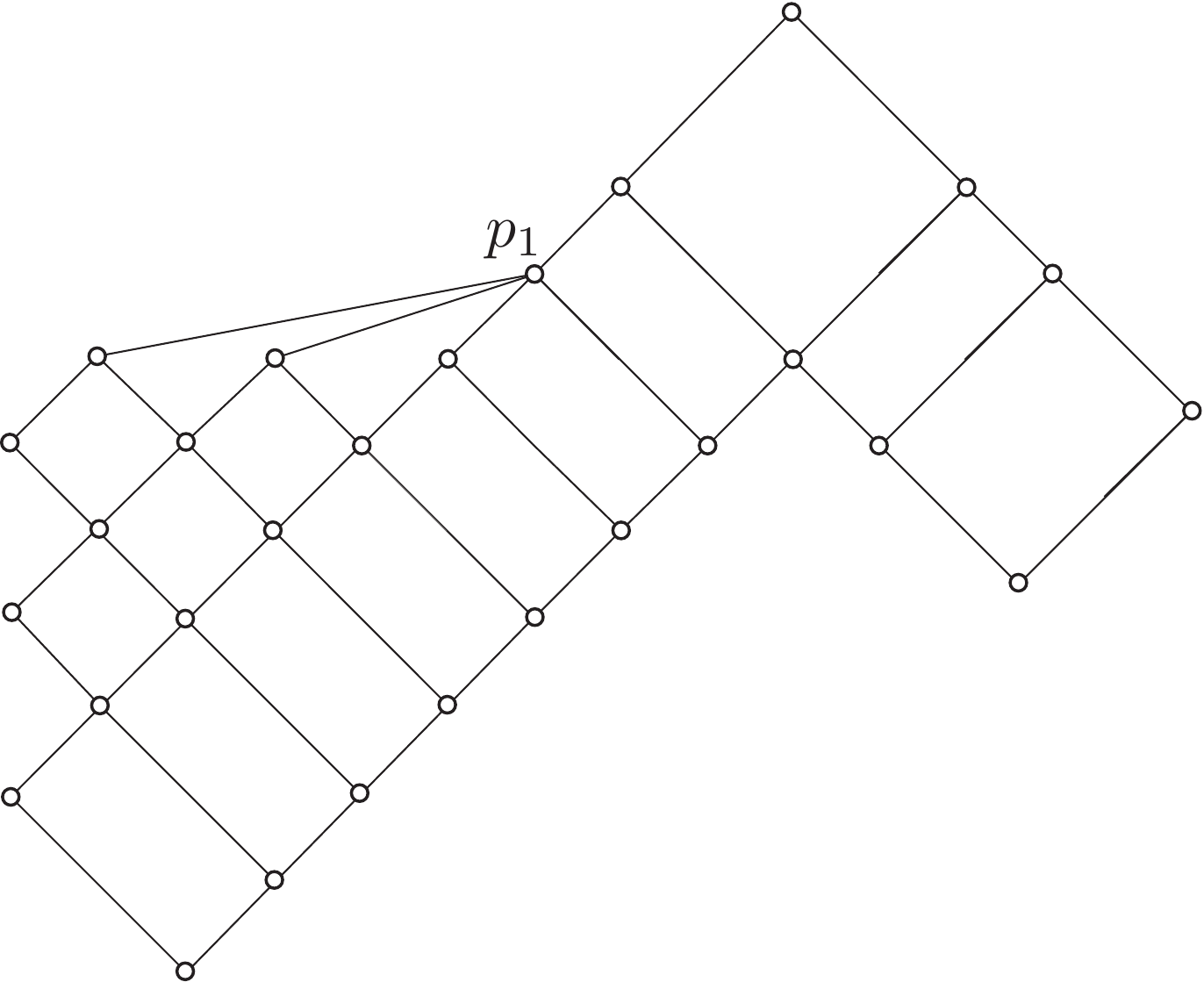}}
\vspace{-50pt}\centerline{\includegraphics[scale=0.55]{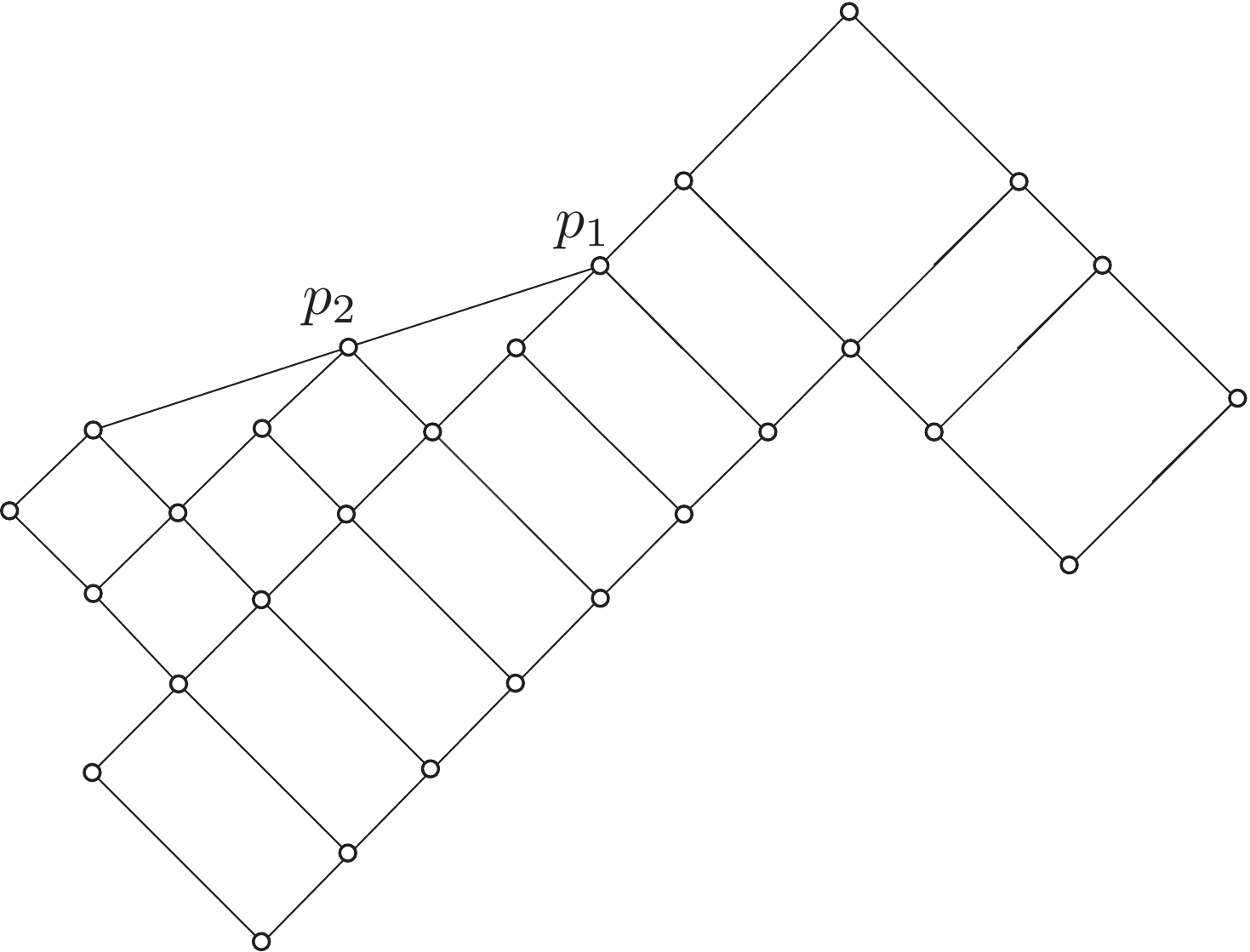}}
\vspace{-54pt}\centerline{\includegraphics[scale=0.55]{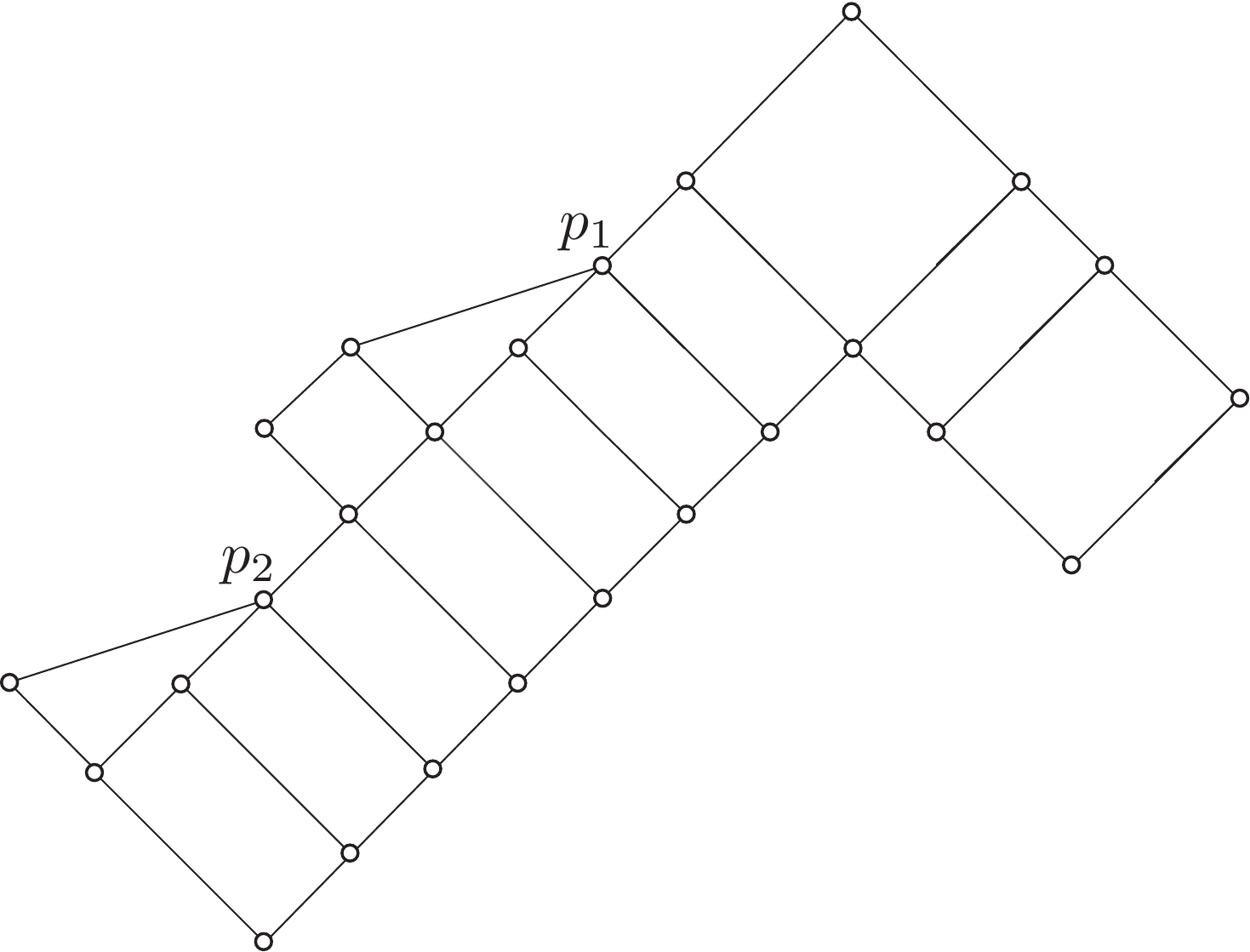}}
\caption{Three variations on the second step}\label{F:protrusionplus}
\end{figure}

We proceed thus in a finite number of steps 
to obtain $K = K_n$, $\bgg(S)$ and $\bgp$.

\begin{lemma}\label{L:ideal}
$K = \id{t}$. 
\end{lemma}

\begin{proof}
For $x \in \id{t}$, we verify that $x \in K$. If $x = t$, the statement is trivial.
Let $x < t$. Then either $x \leq a_l$ or $x \leq a_r$, say,
$x \leq a_l$. Let $x$ be the largest element 
with $x \leq a_l$ and $x \nin \id{t}$. 
Then there is an element $y$ such that $x \prec y \in K$.
If $i$ is the smallest integer with $y \in K_i$,
then $y$ is a protrusion and $x$ will be added with $\New(y)$
(though not necessarily in $K_{i+1}$). 
\end{proof}

More formally, if $K_{n-1} = \Btw(C_{l, n-1}, C_{r, n-1})$, $\bgg_{n-1}$,
and $\bgp_{n-1} = \bgp$ 
have already been defined, 
we define $K_{n}$ as in Section~\ref{S:gammaK}.

If there are no protrusions on $C_{l, n-1}$ and $C_{r, n-1}$, we are done.
If there is a protrusion, say on $C_{l, n-1}$, 
then we choose the largest one, $p$.  
We construct $K_n$ from~$K_{n-1}$ as we have constructed~$K$, namely, 
\[
   K_n = \Btw(C_{l, n-1}, C_{r, n-1}) \uu \New(p).
\]
We then define $\bgg_n$ with \eqref{E:i}--\eqref{E:vi}.

The verification that everything works is the same, 
\emph{mutatis mutandis}, apart from the more complicated notation.

\section{The upper covers of $\bgg(S)$}\label{S:uppercovers}

In this section, we describe the covers of $\bgg(S)$
in the order of join-irreducible congruences of $L[S]$, 
for a tight square~$S$.

Let $G$ denote the set of prime intervals of $L[S]$ listed in 
\eqref{E:seti}--\eqref{E:setiii}. 

\begin{lemma}\label{L:G}
Let $L$ be an SPS lattice. 
Let $S$ be a \emph{tight square}. 
For a prime~interval~$\fp$ of $L[S]$, 
we have $\bgg(S) = \con{\fp}$ if{}f $\fp \in G$.
\end{lemma}

\begin{proof}
If we look at all the prime intervals $\fp$ collapsed by $\bgg(S)$
as listed in \eqref{E:i}--\eqref{E:vi}, then they are either 
listed in \eqref{E:seti}--\eqref{E:setiii} 
or they generate $\bgp$.
\end{proof}

\begin{lemma}\label{L:uppercover}
Let $L$ be an SPS lattice. 
Let $S$ be a \emph{tight square} and 
let $\fp$ be a prime interval of $L$.
Let us assume that $\bgg(S) < \con \fp$
in $L[S]$. Then either $\bga_l(S) \leq \con \fp$
or $\bga_r(S) \leq \con \fp$ in~$L$.
\end{lemma}

\begin{proof}
So let $\bgg(S) < \con \fp$ in $L[S]$ 
for a prime interval $\fp$ of $L[S]$. 
By~Lemmas \ref{L:cproj} and \ref{L:G}, 
there is a sequence of intervals in $L[S]$:
\begin{equation}\tag{S}
   \fp = [e_0,f_0] \cpersp [e_1,f_1] \cpersp 
      \cdots \cpersp [e_n,f_n] = \fq= [z_{l,i}, x_{l,i}] \in G.
\end{equation}
(or symmetrically), using the notation $t = x_{l,0}$, $m = z_{l,0}$.
We can assume that (S) was chosen to minimize $n$. 
In particular, $\cperspup$ and $\cperspdn$ alternate in~(S).

The interval $[e_{n-1}, f_{n-1}]$ 
has a prime subinterval $[e'_{n-1}, f'_{n-1}]$ perspective to~$\fq$. 
Since $G$ is a set of prime interval 
closed under prime perspectivity, 
it follows that $[e_{n-1}, f_{n-1}]$ has $\fq$ as a subinterval
by minimality.

We cannot have $[e'_{n-1}, f'_{n-1}] = [e_{n}, f_{n}]$ 
because this conflicts with the minimality of $n$. So there are two cases to consider.

\emph{Case 1}: $e_{n-1} < e_n$. Then $f_{n-1} = f_n$.
If $[e_n, f_n] = [m,i]$, 
then $e_{n-1} \leq z_{l,1}$ (or symmetrically) 
because $m$ is meet-irreducible. 
Therefore, $\cng z_{l,1} = t(\con{\fp})$ 
and so $\con{\fp} \geq \bga_l$, as claimed.

So we can assume that $[e_n, f_n] = [z_{l,i}, x_{l,i}]$ 
(or symmetrically) with $i \geq 1$. Then $e_{n-1} < z_{l,i}$ 
and so $e_{n-1} \leq z_{l,i+1}$ or $e_{n-1} \leq y_{l,i}$. 
The first possibility contradicts the minimality of $n$, while the second yields that 
$\con{\fp} \geq \bga_r$, as claimed.

\emph{Case 2}: $e_{n-1} = e_n$ and $[e_{n-1}, f_n] \in G$.
Since $[e_{n-1},f_{n-1}] \cperspdn [e_{n},f_{n}]$, it follows that 
$[e_{n-2},f_{n-2}] \cperspup [e_{n-1},f_{n-1}]$.
Therefore, $e_{n-2} \leq x_{l,i} < x_{l,i} < f_{n-2}$
and $f_{n-2} \jj e_{n-1} = f_{n-1}$
(which implies that $x_{l,i}$ precedes the first protrusion),
contradicting the minimality of $n$.
\end{proof}

From Lemma \ref{L:uppercover}, 
we immediately conclude Theorem~\ref{T:uppercover}.

\section{Concluding comments}\label{S:comments}
G.~Gr\"atzer, H. Lakser, and E.\,T. Schmidt \cite{GLS98a} started
the study of planar semimodular lattices and their congruences. 
This field is surveyed in the chapters 
G.~Cz\'edli and G. Gr\"atzer \cite{CGa}
and G. Gr\"atzer \cite{gG14a}
in the book \cite{LTE}.

As it is proved in G.~Gr\"atzer, H. Lakser, and E.\,T. Schmidt \cite{GLS98a},
G. Gr\"atzer and E. Knapp~\cite{GKn09},
G. Gr\"atzer and E.\,T. Schmidt~\cite{GS13b},
every finite distributive lattice $D$ can be represented 
as the congruence lattice of a finite planar semimodular lattice~$L$.

It is crucial that in this result the lattice $L$ 
is not assumed to be slim.
The $\SM 3$ sublattices play a central role in the constructions.
So what happened if we cannot have $\SM 3$ sublattices?

\begin{problem}\label{P:SPS}
Characterize the congruence lattices of SPS lattices.
\end{problem}

\begin{problem}\label{P:SPS3}
Characterize the congruence lattices of patch lattices.
\end{problem}

\begin{problem}\label{P:SPS4}
Characterize the congruence lattices of slim patch lattices.
\end{problem}

Of course, if we drop ``slim'' from SPS, there is nothing to say.
Any finite order can be represented as $\Ji(\Con L)$
of a planar semimodular lattice $L$,
see G. Gr\"atzer, H. Lakser, and E.\,T. Schmidt \cite{GLS98a}, or of a rectangular lattice, see G. Gr\"atzer and E. Knapp~\cite{GKn09}; 
see also G. Gr\"atzer and E.\,T. Schmidt~\cite{GS13b}.

\begin{figure}[hbt]
\centerline{\includegraphics{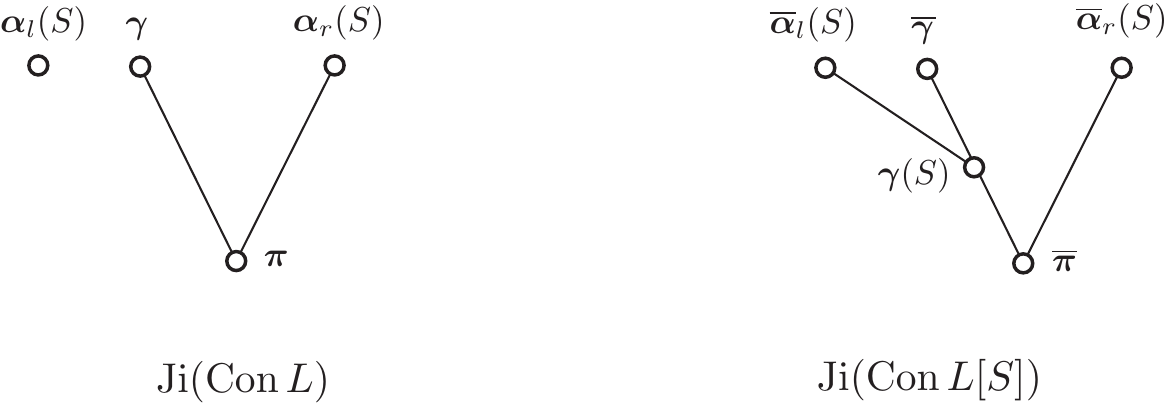}}
\caption{Not suborder}\label{F:notsuborder}
\end{figure}

Look at Figure~\ref{F:notsuborder}. 
It shows the diagrams of $\Ji(\Con L)$ and $\Ji(\Con L[S])$
for the lattice $L$ and covering square $S$ of Figure~\ref{F:forkexample}.
We see that $\Ji(\Con L[S]) = \Ji(\Con L) \uu \set{\bgg(S)}$.
The map $\bga \mapsto \oa$ is an isotone and one-to-one map
of $\Ji(\Con L)$ into $\Ji(\Con L[S])$ but it is not an embedding.
Indeed, $\con{a_l,n} \nleq \con{a_l,t}$ 
but $\consub{L[S]}{a_l,n} \leq \consub{L[S]}{a_l,t}$.

\begin{problem}\label{P:interrelate}
How do $\Ji(\Con L)$ and $\Ji(\Con L[S])$ interrelate?
\end{problem}

\end{document}